\documentclass[12pt]{article}

\usepackage{amsmath,amssymb,amsthm}
\usepackage[all]{xy}

\allowdisplaybreaks
\numberwithin{equation}{section}

\newcommand{\Fg}{\mathfrak{g}}
\newcommand{\Fh}{\mathfrak{h}}

\newcommand{\BC}{\mathbb{C}}
\newcommand{\BR}{\mathbb{R}}

\newcommand{\BZ}{\mathbb{Z}}
\newcommand{\BB}{\mathbb{B}}
\newcommand{\BA}{\mathbb{A}}

\newcommand{\CL}{\mathcal{L}}
\newcommand{\CB}{\mathcal{B}}
\newcommand{\CC}{\mathcal{C}}
\newcommand{\CT}{\mathcal{T}}

\newcommand{\Hom}{\mathop{\rm Hom}\nolimits}
\newcommand{\wt}{\mathop{\rm wt}\nolimits}
\newcommand{\rank}{\mathop{\rm rank}\nolimits}
\newcommand{\dist}{\mathop{\rm dist}\nolimits}
\newcommand{\ch}{\mathop{\rm ch}\nolimits}

\newcommand{\ve}{\varepsilon}
\newcommand{\vp}{\varphi}

\newcommand{\bzero}{\mathbf{0}}
\newcommand{\q}{\mathsf{v}}

\newcommand{\pair}[2]{\langle #1,\,#2 \rangle}

\newcommand{\ud}[1]{\underline{#1}}

\newcommand{\ti}[1]{\tilde{#1}}
\newcommand{\wti}[1]{\widetilde{#1}}

\theoremstyle{plain}
\newtheorem{lem}{Lemma}[section]
\newtheorem{prop}[lem]{Proposition}
\newtheorem{thm}[lem]{Theorem}
\newtheorem{cor}[lem]{Corollary}

\newtheorem{ithm}{Theorem}
\newtheorem{icor}[ithm]{Corollary}

\theoremstyle{definition}
\newtheorem{dfn}[lem]{Definition}

\theoremstyle{remark}
\newtheorem{ex}[lem]{Example}
\newtheorem{rem}[lem]{Remark}

\newenvironment{enu}{%
 \begin{enumerate}%
}{\end{enumerate}}

\begin{document}

\baselineskip=17pt

\title{\Large\bf 
Path model for an extremal weight module \\
over the quantized hyperbolic Kac-Moody algebra \\
of rank 2%
\footnote{2010 Mathematics Subject Classification: 17B37, 17B67, 81R50.}%
}
\author{
Daisuke Sagaki \\ 
 \small Institute of Mathematics, University of Tsukuba, \\
 \small 1-1-1 Tennodai, Tsukuba, Ibaraki 305-8571, Japan \\ 
 \small (e-mail: {\tt sagaki@math.tsukuba.ac.jp})\\[3mm]
and \\[3mm]
Dongxiao Yu \\ 
 \small Graduate School of Pure and Applied Sciences, University of Tsukuba, \\
 \small 1-1-1 Tennodai, Tsukuba, Ibaraki 305-8571, Japan \\
 \small(e-mail: {\tt yudongxiao@math.tsukuba.ac.jp})
}

\date{}

\maketitle

%
\begin{abstract}
Let $\Fg$ be a hyperbolic Kac-Moody algebra of rank 2, 
and set $\lambda=\Lambda_{1} - \Lambda_{2}$, 
where $\Lambda_{1}$, $\Lambda_{2}$ are the fundamental weights. 
Denote by $V(\lambda)$ the extremal weight module of 
extremal weight $\lambda$ with $v_\lambda$ the extremal weight vector, 
and by $\CB(\lambda)$ the crystal basis of $V(\lambda)$ 
with $u_\lambda$ the element corresponding to $v_\lambda$. 
We prove that (i) $\CB(\lambda)$ is connected, 
(ii) the subset $\CB(\lambda)_{\mu}$ of elements of weight $\mu$ in $\CB(\lambda)$ 
is a finite set for every integral weight $\mu$, and $\CB(\lambda)_{\lambda} = \{u_\lambda\}$, 
(iii) every extremal element in $\CB(\lambda)$ is contained in the Weyl group orbit of $u_\lambda$,
(iv) $V(\lambda)$ is irreducible. 
Finally, we prove that the crystal basis $\CB(\lambda)$ is 
isomorphic, as a crystal, to the crystal $\BB(\lambda)$ of 
Lakshmibai-Seshadri paths of shape $\lambda$.
\end{abstract}
%
%
\section{Introduction.} 
\label{sec:intro}
In this paper, we study the structure of the extremal weight module $V(\lambda)$ of
extremal weight $\lambda := \Lambda_{1}-\Lambda_{2}$ over the quantized 
universal enveloping algebra associated to a hyperbolic Kac-Moody algebra of 
rank $2$, where $\Lambda_{1}$, $\Lambda_{2}$ are the fundamental weights, 
and then prove that the crystal basis $\CB(\lambda)$ of $V(\lambda)$ 
is isomorphic, as a crystal, to the crystal of 
Lakshmibai-Seshadri (LS for short) paths of shape $\lambda$.  

Let us explain the background and motivation of this paper. 
Let $\Fg$ be a symmetrizable Kac-Moody algebra over $\BC$ 
with $P$ the integral weight lattice, and 
$U_{\q}(\Fg)$ the quantized universal enveloping algebra 
over $\BC(\q)$ associated to $\Fg$.
The extremal weight module $V(\mu)$ of extremal weight $\mu \in P$ 
is the integrable $U_{\q}(\Fg)$-module 
generated by a single element $v_\mu$ with the defining relation 
that $v_\mu$ is an extremal weight vector of weight $\mu$. 
This module was introduced by Kashiwara \cite[Proposition 8.2.2]{o4} 
as a natural generalization of integrable highest (or lowest) weight modules; 
he also proved that $V(\mu)$ has a crystal basis $\CB(\mu)$. 
We know from \cite[Proposition~8.2.2 (iv) and (v)]{o4} that 
$V(\mu) \cong V(w\mu)$ as $U_{\q}(\Fg)$-modules, 
and $\CB(\mu) \cong \CB(w\mu)$ as crystals for all $\mu \in P$ 
and $w \in W$, where $W$ is the Weyl group of $\Fg$.
Also, we know from the comment at the end of \cite[\S8.2]{o4} 
that if $\mu \in P$ is dominant (resp., antidominant), 
then $V(\mu)$ is isomorphic to the  integrable highest (resp., lowest) weight module of 
highest (resp., lowest) weight $\mu$, 
and $\CB(\mu)$ is isomorphic to its crystal basis. 
So, we are interested in those $\mu \in P$ such that
\begin{equation} \label{20171130(1)}
\begin{split}
& \text{any element of $W\mu$ is neither dominant nor antidominant} \\
& \text{(or, no element of $W\mu$ is either dominant or antidominant)}. 
\end{split}
\end{equation} 
If $\Fg$ is of finite type, then there is no $\mu \in P$ satisfying 
the condition \eqref{20171130(1)}; it is well-known that $W\mu$ contains 
a (unique) dominant integral weight for every $\mu \in P$.
Assume that $\Fg$ is of affine type, and let $c$ be 
the canonical central element of $\Fg$. Then, $\mu \in P$ satisfies 
the condition \eqref{20171130(1)} if and only if $\mu$ is level-zero, 
that is, ($\mu \neq 0$, and) $\pair{\mu}{c} = 0$. 
In \cite{o5} and \cite{BN}, they deeply studied the basic structure of 
$V(\mu)$ and $\CB(\mu)$ for level-zero $\mu \in P$. 
Using results in \cite{o5} and \cite{BN}, Naito and Sagaki 
proved in \cite{o7} and \cite{o11} that if $\mu$ is 
a positive integer multiple of 
a level-zero fundamental weight, then 
the crystal basis $\CB(\mu)$ is isomorphic to
the crystal of LS paths of shape $\mu$; 
for the details on LS paths, see \S\ref{subsec:LS} below. 
After that, in \cite{o15}, they introduced semi-infinite LS paths 
in terms of the semi-infinite Bruhat order on the affine Weyl group, 
and proved that for level-zero dominant $\mu \in P$, 
the crystal basis $\CB(\mu)$ is 
isomorphic, as a crystal, to the crystal of semi-infinite LS paths of shape $\mu$. 
Thus we have already finished the basic study of the structure of 
$V(\mu)$ and $\CB(\mu)$, and given a path model for $\CB(\mu)$ 
in the finite and affine cases. 

In this paper, we consider the case where $\Fg = \Fg(A)$ is 
a hyperbolic Kac-Moody algebra of rank $2$ with Cartan matrix 
\begin{equation*}
A = \begin{pmatrix}
 2 & -a_1 \\
-a_2 & 2
\end{pmatrix},
\qquad \text{where $a_1,\,a_2 \in \BZ_{>0}$ with $a_1a_2 >4$}.
\end{equation*}
In \cite[Proposition~3.1.1]{o1}, it was proved that
\begin{equation*}
\lambda = \Lambda_1 - \Lambda_2
\end{equation*}
satisfies the condition \eqref{20171130(1)} if $a_1, a_2 \neq 1$. 
We prove the following theorems and corollary.

\begin{ithm}[Theorem~\ref{n7}] \label{ithm1}
The crystal graph of $\CB(\lambda)$ is connected.
\end{ithm}

\begin{icor}[Corollary~\ref{n14}] \label{icor2}
For every $\mu \in P$, the subset $\CB(\lambda)_{\mu}$ of 
elements of weight $\mu$ in $\CB(\lambda)$ is a finite set. 
In particular, $\CB(\lambda)_{\lambda} = \{u_{\lambda}\}$, 
where $u_\lambda$ is the element of $\CB(\lambda)$ corresponding to 
the extremal weight vector $v_\lambda \in V(\lambda)$.
\end{icor}

Since $\CB(\lambda)$ is a normal crystal (in the sense of Definition~\ref{normal}), 
$\CB(\lambda)$ has a canonical action $S_w \ (w \in W)$ of 
the Weyl group (see \S\ref{subsec:crystal}). 
Then, $u_\lambda \in \CB(\lambda)$ is an extremal element of weight $\lambda$. 

\begin{ithm}[Theorem~\ref{n10}] \label{ithm2} \mbox{}
\begin{enu}
\item Let $x, y \in W$. 
Then, $S_x u_\lambda = S_y u_\lambda$ if and only if $x\lambda = y\lambda$.

\item If $b \in \CB(\lambda)$ is extremal, 
then there exists  $w \in W$ such that $b = S_wu_\lambda$.
\end{enu}
\end{ithm}

\begin{ithm}[Theorem~\ref{n1124(1)}] \label{ithm3}
The extremal weight module $V(\lambda)$ of extremal weight $\lambda$ is irreducible.
\end{ithm}


Finally, we prove the following isomorphism theorem. 
\begin{ithm}[Theorem~\ref{20171126(4)}] \label{ithm4}
The crystal basis $\CB(\lambda)$ of 
the extremal weight module $V(\lambda)$ of 
extremal weight $\lambda$ is isomorphic, as a crystal, to
the crystal $\BB(\lambda)$ of LS paths of shape $\lambda$.  
\end{ithm}

By Corollary~\ref{icor2}, 
we can define the character of $V(\lambda)$ by:
\begin{equation*}
\ch V(\lambda)=\sum_{\mu \in P} \dim V(\lambda)_{\mu} e(\mu);
\end{equation*}
By Theorem~\ref{ithm4} above, we see that 
$\dim V(\lambda)_{\mu}$ is equal to the number of 
elements of weight $\mu$ in the crystals $\CB(\lambda) \cong \BB(\lambda)$. 
In \cite[Theorem 4.1.2]{o1}, she gave a combinatorial description of 
the crystal $\BB(\lambda)$. Combining these results, 
we can obtain a combinatorial formula of $\ch V(\lambda)$ 
(though it may be not very explicit).
In \cite{NS16}, \cite{LNSSS3}, and \cite{NNS}, 
they proved that in the untwisted affine case, 
the graded characters of certain subquotients of extremal weight modules
can be written in terms of specializations of symmetric and nonsymmetric 
Macdonald polynomials. We expect that also in the hyperbolic case, 
some interesting ``symmetric'' series arise as the characters of 
extremal weight modules.

This paper is organized as follows. 
In Section~\ref{sec:pre}, we fix our notation, and recall the definitions and 
basic properties of extremal weight modules and 
their crystal bases, and also of LS paths. 
In Subsection~\ref{subsec:main}, we state our main results 
(the four theorems and one corollary above). 
Subsections~\ref{subsec:prf1}, \ref{subsec:prf2}, and \ref{subsec:prf3} are 
devoted to proofs of Theorems~\ref{n7}, \ref{n10}, and \ref{n1124(1)}, respectively. 
In Section~\ref{sec:prf4}, we prove the isomorphism theorem (Theorem~\ref{20171126(4)})
in more general setting, after showing ``similarities'' of the crystals 
(Proposition~\ref{prop:s1} and \ref{prop:s2}).
%
%
\section{Preliminaries.}
\label{sec:pre}
%
%
\subsection{Kac-Moody algebras.}
\label{subsec:km}

Let $A = (a_{ij})_{i, j \in I}$ be 
a symmetrizable generalized Cartan matrix 
with $I$ the finite index set. Let $\Fg = \Fg(A)$ be 
the Kac-Moody algebra  associated to $A$ over $\BC$. 
Denote by $\Fh$ the Cartan subalgebra of $\Fg$ (with $\dim \Fh=2|I|-\rank A$), 
$\{\alpha_i \mid i \in I\} \subset \Fh^{\ast} : = \Hom_{\BC}(\Fh, \BC)$ 
the set of (linearly independent) simple roots, 
and $\{\alpha_i^{\vee} \mid i \in I\} \subset \Fh$ 
the set of (linearly independent) simple coroots.
We set $Q_{+} := \sum_{i \in I}\BZ_{\geq 0}\alpha_i$. 
Denote by $W = \langle r_i \mid i \in I \rangle$ 
the Weyl group of $\Fg$, where $r_i$ is the simple reflection in $\alpha_i$ for $i \in I$.
Let $\Lambda_i \in \Fh^{\ast}, i \in I$, be 
the fundamental weights for $\Fg$, i.e., 
$\pair{\Lambda_i}{\alpha_j^{\vee}} = \delta_{i, j}$ 
for $i,\,j \in I$. Take an integral weight lattice $P$ 
containing $\alpha_{i}$, $i \in I$, and $\Lambda_{i}$, $i \in I$.
%
%
\subsection{Crystal bases and crystals.}
\label{subsec:crystal}

For the definitions of crystal bases and crystals, see \cite{o2} and \cite{o8}; 
our convention of the tensor product rule of crystals is 
the same as that in these references. Namely, for crystals $\CB_{1}$ and $\CB_{2}$, 
we define their tensor product $\CB_{1} \otimes \CB_{2}$ as follows: 
for $b_{1} \in \CB_{1}$, $b_{2} \in \CB_{2}$, and $i \in I$, 
\begin{equation*}
\begin{split}
& \wt(b_{1} \otimes b_{2}):=\wt(b_{1}) + \wt(b_{2}), \\[2mm]
& \ti{e}_{i}(b_{1} \otimes b_{2}):=
  \begin{cases}
  \ti{e}_{i}b_{1} \otimes b_{2} & \text{if $\vp_{i}(b_{1}) \ge \ve_{i}(b_{2})$}, \\
  b_{1} \otimes \ti{e}_{i}b_{2} & \text{if $\vp_{i}(b_{1}) < \ve_{i}(b_{2})$},
  \end{cases} \\[2mm]
& \ti{f}_{i}(b_{1} \otimes b_{2}):=
  \begin{cases}
  \ti{f}_{i}b_{1} \otimes b_{2} & \text{if $\vp_{i}(b_{1}) > \ve_{i}(b_{2})$}, \\
  b_{1} \otimes \ti{f}_{i}b_{2} & \text{if $\vp_{i}(b_{1}) \le \ve_{i}(b_{2})$},
  \end{cases} \\[2mm]
& \ve_{i}(b_{1} \otimes b_{2}):=
  \max \bigl\{ \ve_{i}(b_{1}), \, \ve_{i}(b_{2})-\pair{\wt(b_{1})}{\alpha_{i}^{\vee}} \bigr\}, \\[2mm]
& \vp_{i}(b_{1} \otimes b_{2}):=
  \max \bigl\{ \vp_{i}(b_{2}), \, \vp_{i}(b_{1})+\pair{\wt(b_{2})}{\alpha_{i}^{\vee}} \bigr\}. 
\end{split}
\end{equation*}

Let $\CB$ be a crystal, and let $\ti{e}_i$ and $\ti{f}_i$, $i \in I$, 
be the Kashiwara operators for $\CB$. 
For $b \in \CB$ and $i \in I$, we set $\ti{e}_i^{\max} b := \ti{e}_i^{\ve_i(b)}b$ 
if $\ve_i(b) = \max\{n \geq 0 \mid \ti{e}_i^n b \neq \bzero\}$, 
where $\bzero$ is an extra element not contained in any crystal. 
Similarly, we set $\ti{f}_i^{\max}b := \ti{f}_i^{\vp_i(b)}b$ 
if $\vp_i(b) = \max\{n \geq 0 \mid \ti{f}_i^n b \neq \bzero\}$.

\begin{dfn}[{see \cite[page 389]{o4} and \cite[page 182]{o2}}] \label{normal}
A crystal $\CB$ is said to be \emph{normal} 
if it satisfies the following condition for every $J \subset I$ 
such that the Levi subalgebra $\Fg_J$ of $\Fg$ corresponding to $J$ is finite-dimensional: 
if we regard $\CB$ as a crystal for 
$U_{\q}(\Fg_J)$ by restriction, 
then it is isomorphic to the crystal basis of 
a finite-dimensional $U_{\q}(\Fg_J)$-module.
\end{dfn}

\begin{rem} \label{20171125(2)} \mbox{}
\begin{enu}
\item 
If $\CB$ is a normal crystal, then 
$\ve_i(b) = \max\{n \geq 0 \mid \ti{e}_i^n b \neq \bzero\}$ and 
$\vp_i(b) = \max\{n \geq 0 \mid \ti{f}_i^n b \neq \bzero\}$ 
for all $b \in \CB$ and $i \in I$. 

\item 
In the case that $\Fg$ is a hyperbolic Kac-Moody algebra of rank $2$, 
the converse of part (1) also holds since $\Fg$ is infinite-dimensional. 
In particular, the main crystals in \S\ref{sec:main} below 
are obviously normal.
\end{enu}
\end{rem}

We know from \cite[\S7]{o4} (see also \cite[Theorem 11.1]{o2}) 
that a normal crystal $\CB$ has an action of the Weyl group $W$ as follows. 
For $i \in I$ and $b \in \CB$, we set
\begin{equation} \label{Si}
S_ib := 
\begin{cases}
\ti{f}_i^{ \pair{\wt(b)}{\alpha_i^{\vee}} }b 
 & \text{if $\pair{\wt(b)}{\alpha_i^{\vee}} \geq 0$}, \\[1mm]
\ti{e}_i^{ -\pair{\wt(b)}{\alpha_i^{\vee}} }b 
 & \text{if $\pair{\wt(b)}{\alpha_i^{\vee}} \leq 0$}.
\end{cases}
\end{equation} 
Then, for $w \in W$, we set 
$S_w := S_{i_1} \cdots S_{i_k}$ if $w = r_{i_1} \cdots r_{i_k}$. 
Notice that $\wt(S_w b) = w \wt(b)$ for $w \in W$ and $b \in \CB$.

\begin{dfn}
An element $b$ of a normal crystal $\CB$ is said to be \emph{extremal} 
if for each $w \in W$ and $i \in I$, 
$\ti{e}_i(S_w b) = \bzero$ (resp., $\ti{f}_i(S_w b) = \bzero$) 
if $\pair{\wt(S_w b)}{\alpha_i^{\vee}} \geq 0$ (resp., $\leq 0$).
\end{dfn}

Now, let $\CB(\infty)$ (resp., $\CB(-\infty)$) be 
the crystal basis of the negative part $U_{\q}^-(\Fg)$ 
(resp., the positive part $U_{\q}^+(\Fg)$) of 
the quantized universal enveloping algebra $U_{\q}(\Fg)$ over $\BC(\q)$ associated to $\Fg$, 
and let $u_\infty \in \CB(\infty)$ (resp., $u_{-\infty} \in \CB(-\infty)$) be 
the element corresponding to $1 \in U_{\q}^-(\Fg)$ (resp., $1 \in U_{\q}^+(\Fg)$). 
Denote by $\ti{e}_i$ and $\ti{f}_i$, $i \in I$, 
the raising and lowering Kashiwara operators on $\CB(\pm \infty)$, respectively. 
For $i \in I$, we define $\ve_i,\,\vp_i : \CB(\infty) \rightarrow \BZ$ and 
$\ve_i,\,\vp_i : \CB(-\infty) \rightarrow \BZ$ by
\begin{equation*}
\begin{split}
& \ve_i(b) := \max\{n \geq 0 \mid \ti{e}_i^n b \neq \bzero\}, \quad 
  \vp_i(b) := \ve_i(b) + \pair{\wt(b)}{\alpha_i^{\vee}} \quad \text{for $b \in \CB(\infty)$}, \\[2mm]
& \vp_i(b) := \max\{n \geq 0 \mid \ti{f}_i^n b \neq \bzero\}, \quad 
  \ve_i(b) := \vp_i(b) - \pair{\wt(b)}{\alpha_i^{\vee}} \quad \text{for $b \in \CB(-\infty)$},
\end{split}
\end{equation*}
respectively. Denote by $\ast : \CB(\pm \infty) \rightarrow \CB(\pm \infty)$ 
the $\ast$-operator on $\CB(\pm \infty)$, 
which is induced from a $\BC(\q)$-algebra antiautomorphism 
$\ast : U_{\q}(\Fg) \rightarrow U_{\q}(\Fg)$ 
(see \cite[Theorem~2.1.1]{o6} and \cite[\S8.3]{o2}). 
We see that $\wt(b^{\ast}) = \wt(b)$ for all $b \in \CB(\pm \infty)$.
The next lemma (which is likely well-known to experts)
follows immediately from the fact that 
$\ti{f}_i^k u_{\infty}$ (resp., $\ti{e}_i^k u_{-\infty}$) is 
a unique element of weight $-k\alpha_i$ (resp., $k\alpha_i$) 
in $\CB(\infty)$ (resp., $\CB(-\infty)$).

\begin{lem} \label{n11}
We have $(\ti{f}_i^k u_{\infty})^\ast = \ti{f}_i^{k}u_{\infty}$ 
for all $k \in \BZ_{\geq 0}$ and $i \in I$. Similarly, 
we have  $(\ti{e}_i^k u_{-\infty})^\ast = \ti{e}_i^{k}u_{-\infty}$ 
for all $k \in \BZ_{\geq 0}$ and $i \in I$.
\end{lem}

For $\mu \in P$, denote by $\CT_\mu = \{t_\mu\}$ 
the crystal consisting of a single element $t_\mu$ such that 
\begin{equation*}
\begin{split}
& \wt(t_\mu) = \mu, \\
& \ti{e}_i t_{\mu} = \ti{f}_i t_{\mu} = \bzero \quad \text{for $i \in I$}, \\
& \ve_i(t_\mu) = \vp_i(t_\mu) = -\infty \quad \text{for  $i \in I$}.
\end{split}
\end{equation*}
%
%
\subsection{Crystal bases of extremal weight modules.}
\label{subsec:extmod}

Let $\mu \in P$ be an arbitrary integral weight. 
The \emph{extremal weight module} $V(\mu)$ of extremal weight $\mu$ is, 
by definition, the integrable $U_{\q}(\Fg)$-module generated 
by a single element $v_\mu$ with the defining relation that 
$v_\mu$ is an extremal weight vector of weight $\mu$ 
in the sense of \cite[Definition~8.1.2]{o4}. 
We know from \cite[Proposition~8.2.2]{o4} that 
$V(\mu)$ has a crystal basis $(\CL(\mu), \CB(\mu))$, and 
a global basis $\{G(b) \mid b \in \CB(\mu)\}$; recall that 
\begin{equation} \label{20171126(1)}
V(\mu) = \bigoplus_{b \in \CB(\mu)}\BC(\q)G(b), \hspace{8mm} 
\CL(\mu) = \bigoplus_{b \in \CB(\mu)} \BA_0G(b),
\end{equation}
where $\BA_0 := \{f(\q) \in \BC(\q) \mid \text{$f(\q)$ is regular at $\q = 0$}\}$.

\begin{rem} \label{20171126(2)}
We see from \cite[Proposition 8.2.2 (iv) and (v)]{o4} that 
$V(\mu) \cong V(w\mu)$ as $U_{\q}(\Fg)$-modules, 
and $\CB(\mu) \cong \CB(w\mu)$ as crystals for all $\mu \in P$ and $w \in W$. 
Also, we know from the comment at the end of \cite[\S8.2]{o4} that 
if $\mu \in P$ is dominant (resp., antidominant), 
then $V(\mu)$ is isomorphic, as a $U_{\q}(\Fg)$-module, 
to the integrable highest (resp., lowest) weight module 
of highest (resp., lowest) weight $\mu$, 
and $\CB(\mu)$ is isomorphic, as a crystal, to its crystal basis. 
So, we are interested in those $\mu \in P$ such that
\begin{equation}\label{eq31}
\begin{split}
& \text{any element of $W\mu$ is neither dominant nor antidominant} \\
& \text{(or, no element of $W\mu$ is either dominant or antidominant).}
\end{split}
\end{equation} 
\end{rem}

Now, the crystal basis $\CB(\mu)$ of $V(\mu)$ 
can be realized (as a crystal) as follows. We set
\begin{equation*}
\CB := \bigsqcup_{\mu \in P} \CB(\infty) \otimes \CT_{\mu} \otimes \CB(-\infty);
\end{equation*}
in fact, $\CB$ is isomorphic, as a crystal, to 
the crystal basis $\CB(\wti{U}_{\q}(\Fg))$ of 
the modified quantized universal enveloping algebra $\wti{U}_{\q}(\Fg)$ 
associated to $\Fg$ (see \cite[Theorem 3.1.1]{o4}). 
Denote by $\ast : \CB \rightarrow \CB$ the $\ast$-operation on $\CB$, 
which is induced from a $\BC(\q)$-algebra antiautomorphism 
$\ast : \wti{U}_{\q}(\Fg) \rightarrow \wti{U}_{\q}(\Fg)$ (see \cite[Theorem~4.3.2]{o4}); 
we know from \cite[Corollary~4.3.3]{o4} 
that for $b_1 \in \CB(\infty)$, $b_2 \in \CB(-\infty)$, and $\mu \in P$,
\begin{equation} \label{eq26}
(b_1 \otimes t_\mu \otimes b_2)^{\ast} = 
 b_1^\ast \otimes t_{-\mu-\wt(b_1)-\wt(b_2)} \otimes b_2^\ast; 
\end{equation}
\begin{rem} \label{rem:wt}
The weight of 
$(b_1 \otimes t_\mu \otimes b_2)^{\ast}$ is equal to $-\mu$ for all
$b_1 \in \CB(\infty)$ and $b_2 \in \CB(-\infty)$ since 
$\wt (b_1^{\ast}) = \wt(b_1)$ and $\wt (b_2^{\ast}) = \wt(b_2)$.
\end{rem}

Because $\CB$ is a normal crystal by \cite[\S2.1 and Theorem 3.1.1]{o4}, 
$\CB$ has the action of the Weyl group $W$ (see \S\ref{subsec:crystal}). 
We know the following proposition from \cite[Proposition 8.2.2 (and Theorem 3.1.1)]{o4}.

\begin{prop}\label{20171126(3)}
For $\mu \in P$, the subset
\begin{equation} \label{eq07}
\bigl\{ 
 b \in \CB(\infty) \otimes \CT_{\mu} \otimes \CB(-\infty) \mid 
 \text{\rm $b^{\ast}$ is extremal} \bigr\}
\end{equation}
is a subcrystal of $\CB(\infty) \otimes \CT_{\mu} \otimes \CB(-\infty)$, 
and is isomorphic, as a crystal, to 
the crystal basis $\CB(\mu)$ of 
the extremal weight module $V(\mu)$ of extremal weight $\mu$.
\end{prop}

In the following, we identify the crystal basis $\CB(\mu)$ 
with the subcrystal \eqref{eq07} of 
$\CB(\infty) \otimes \CT_{\mu} \otimes \CB(-\infty)$. We set 
\begin{equation}\label{eq30}
u_\mu := u_{\infty} \otimes t_\mu \otimes u_{-\infty} 
\in \CB(\infty) \otimes \CT_{\mu} \otimes \CB(-\infty)
\end{equation}
Then, $u_\mu$ is an extremal element of weight $\mu$ contained in $\CB(\mu)$, 
which corresponds to the extremal weight vector $v_\mu \in V(\mu)$ 
in the sense that $v_\mu = G(u_\mu)$.
%
%
\subsection{Lakshmibai-Seshadri paths.}
\label{subsec:LS}

Let us recall the definition of 
Lakshmibai-Seshadri paths (LS paths for short) 
from \cite[\S4]{o10} (see also \cite[\S2.2]{o1}). 
Let $\mu \in P$ be an arbitrary integral weight. 

\begin{dfn} \label{20171124(2)}
For $\nu,\,\nu' \in W\mu$, 
we write $\nu \ge \nu'$ if there exist a sequence 
$\nu=\xi_{0},\,\xi_{1},\,\dots$, $\xi_{p}=\nu'$ of elements in $W\mu$ and 
a sequence $\beta_{1},\,\dots,\,\beta_{p}$ of positive real roots 
such that $\xi_{q}=r_{\beta_{q}}(\xi_{q-1})$ and 
$\pair{\xi_{q-1}}{\beta_{q}^{\vee}} < 0$ for each $q=1,\,2,\,\dots,\,p$, 
where for a positive real root $\beta$, $r_{\beta} \in W$ denotes  
the reflection in $\beta$, and $\beta^{\vee}$ denotes the dual root of $\beta$. 
If $\nu \ge \nu'$, then we define $\dist(\nu,\nu')$ to 
be the maximal length $p$ of all possible such sequences 
$\nu=\xi_{0},\,\xi_{1},\,\dots,\,\xi_{p}=\nu'$ for $(\nu,\nu')$.
\end{dfn}

\begin{dfn} \label{20171124(3)}
Let $\nu,\,\nu' \in W\mu$ with $\nu > \nu'$, 
and let $0 < \sigma < 1$ be a rational number. 
A $\sigma$-chain for $(\nu,\nu')$ is
a sequence $\nu=\xi_{0} > \xi_{1} > 
\dots > \xi_{p}=\nu'$ of elements in $W\mu$ 
such that $\dist(\xi_{q-1},\xi_{q})=1$ and 
$\sigma\pair{\xi_{q-1}}{\beta_{q}^{\vee}} \in \BZ_{< 0}$ 
for all $q=1,\,2,\,\dots,\,p$, where $\beta_{q}$ is 
the (unique) positive real root such that $\xi_q = r_{\beta_q}\xi_{q-1}$.
\end{dfn}

\begin{dfn} \label{dfn:LS}
An \emph{LS path of shape $\mu$} is a pair $\pi=(\ud{\nu}\,;\,\ud{\sigma})$ 
of a sequence $\ud{\nu}:\nu_{1} > \nu_{2} > \cdots > \nu_{s}$ 
of elements in $W\mu$ and a sequence 
$\ud{\sigma}:0=\sigma_{0} < \sigma_{1} < \cdots < \sigma_{s}=1$ of 
rational numbers such that 
there exists a $\sigma_{u}$-chain for $(\nu_{u},\,\nu_{u+1})$ 
for all $u=1,\,2,\,\dots,\,s-1$.
\end{dfn}

Denote by $\BB(\mu)$ the set of LS paths of shape $\mu$. 
We identify $\pi=(\ud{\nu}\,;\,\ud{\sigma}) \in \BB(\mu)$ 
(as in Definition~\ref{dfn:LS}) with
the following piecewise-linear continuous map 
$\pi:[0,1] \rightarrow \BR \otimes_{\BZ} P$: 
\begin{equation*}
\pi(t)=\sum_{k=1}^{u-1}
(\sigma_{k}-\sigma_{k-1})\nu_{k}+
(t-\sigma_{u-1})\nu_{u} \quad \text{for \ }
\sigma_{u-1} \le t \le \sigma_{u}, \  1 \le u \le s.
\end{equation*}

\begin{rem} \label{rem:str}
We see from the definition of LS paths that 
$\pi_{\nu}:=(\nu ; 0,1) \in \BB(\mu)$ for every $\nu \in W\mu$, 
which corresponds to the straight line $\pi_{\nu}(t)=t\nu$ for $t \in [0,1]$. 
\end{rem}

Now, we endow $\BB(\mu)$ with a crystal structure as follows. 
First, we define $\wt(\pi) := \pi(1)$ for $\pi \in \BB(\mu)$; 
we know from \cite[Lemma 4.5]{o10} that $\pi(1) \in P$. 
Next, for $\pi \in \BB(\mu)$ and $i \in I$,  we define
\begin{equation} \label{eq:eq105}
H_i^{\pi}(t) := \pair{\pi(t)}{\alpha_{i}^{\vee}} \quad 
\text{for $t \in [0, 1]$}, \qquad 
m_i^{\pi} := \min \{H_{i}^{\pi}(t) \mid t \in [0, 1]\}.
\end{equation}
We know from \cite[Lemma 4.5]{o10} that
\begin{equation}\label{eq:eq106}
\text{\rm all local minimal values of $H_i^{\pi}(t)$ are are integers}; 
\end{equation}
in particular, $m_i^{\pi}$ is a nonpositive integer, 
and $H_i^{\pi}(1) - m_i^{\pi}$ is a nonnegative integer. 
We define $\ti{e}_i\pi$ as follows: 
If $m_i^{\pi} = 0$, then we set $\ti{e}_i \pi := \bzero$. 
If $m_i^{\pi} \leq -1$, then we set
\begin{equation} \label{eq:eq107}
\begin{split}
t_1 &:= \min \{t \in [0, 1] \mid H_i^{\pi}(t) = m_i^{\pi}\}, \\
t_0 &:= \max \{t \in [0, t_1] \mid H_i^{\pi}(t) = m_i^{\pi} + 1\};
\end{split}
\end{equation}
we see by \eqref{eq:eq106} that 
$H_i^{\pi}(t)$ is strictly decreasing on $[t_0, t_1]$. We define
\begin{equation*}
(\ti{e}_i \pi)(t) := 
\begin{cases}
 \pi(t) & \text{if $0 \leq t \leq t_0$}, \\
 r_i(\pi(t) - \pi(t_0)) + \pi(t_0)  & \text{if $t_0 \leq t \leq t_1$}, \\
 \pi(t) + \alpha_i & \text{if $t_1 \leq t \leq 1$};
\end{cases}
\end{equation*}
we know from \cite[\S4]{o10} that $\ti{e}_i\pi \in \BB(\lambda)$.
Similarly, we define $\ti{f}_i \pi$ as follows: 
If $H_i^{\pi}(1) - m_i^{\pi} = 0$, 
then we set $\ti{f}_i \pi := \bzero$. 
If $H_i^{\pi}(1) - m_i^{\pi} \geq 1$, then we set
\begin{equation}\label{eq:eq109}
\begin{split}
t_0 &:= \max \{t \in [0, 1] \mid H_i^{\pi}(t) = m_i^{\pi}\},\\
t_1 &:= \min \{t \in [t_0, 1] \mid H_i^{\pi}(t) = m_i^{\pi} + 1 \};
\end{split}
\end{equation}
we see by \eqref{eq:eq106} that $H_i^{\pi}(t)$ is 
strictly increasing on $[t_0, t_1]$. We define
\begin{equation*}
(\ti{f}_i \pi)(t) := 
\begin{cases}
\pi(t)  & \text{if $0 \leq t \leq t_0$}, \\
r_i(\pi(t) - \pi(t_0)) + \pi(t_0)  & \text{if $t_0 \leq t \leq t_1$}, \\
\pi(t) - \alpha_i  &\text{if $t_1 \leq t \leq 1$};
\end{cases}
\end{equation*}
we know from \cite[\S4]{o10} that $\ti{f}_i\pi \in \BB(\mu)$. 
We set $\ti{e}_i\bzero = \ti{f}_i\bzero := \bzero$ for $i \in I$ by convention.
Finally, for $\pi \in \BB(\mu)$ and $i \in I$, we set 
\begin{equation*}
\ve_i(\pi ) := \max \{n \in \BZ_{\geq 0} \mid \ti{e}_i^{n}\pi \neq \bzero\}, \qquad 
\vp_i(\pi ) := \max \{n \in \BZ_{\geq 0} \mid \ti{f}_i^{n}\pi \neq \bzero\};
\end{equation*}
we know from \cite[\S2]{o10} that 
$\ve_i(\pi) = - m_i^{\pi}$ and $\vp_i(\pi) = H_i^{\pi}(1) - m_i^{\pi}$.

\begin{thm}[{\cite[\S2 and \S4]{o10}}] \label{n8}
The set $\BB(\mu)$, together with the maps 
$\wt : \BB(\mu) \rightarrow P$,  
$\ti{e}_i,\,\ti{f}_i : \BB(\mu) \cup \{\bzero\} \rightarrow \BB(\mu) \cup \{\bzero\}$, $i \in I$, 
and $\ve_i,\,\vp_i : \BB(\mu) \rightarrow \BZ_{\geq 0}$, $i \in I$, becomes a crystal.
\end{thm}

Since $\BB(\mu)$ is a normal crystal, 
$\BB(\mu)$ has the action of the Weyl group $W$ as mentioned in \S\ref{subsec:crystal} 
(see also \cite[Theorem~8.1]{o10}).
We can easily show the next lemma by induction on the length of $w \in W$.
\begin{lem} \label{20141124(5)}
For $w \in W$, we have $S_{w}\pi_{\mu} = \pi_{w\mu}$. 
In particular, $\pi_{\mu}$ is an extremal element of weight $\mu$. 
\end{lem}
%
%
\section{Main results.}
\label{sec:main}
%
%
\subsection{Theorems and corollaries.}
\label{subsec:main}

Throughout this section, we assume that $I = \{1,\,2\}$, and 
\begin{equation} \label{eq:eq07}
A = 
 \begin{pmatrix} 
 2 & -a_1 \\
 -a_2 & 2
\end{pmatrix}, \qquad
\text{where $a_1,\,a_2 \in \BZ_{\ge 2}$ with $a_1a_2 > 4$}; 
\end{equation}
we set
\begin{equation} \label{eq:lam}
\lambda := \Lambda_1 - \Lambda_2.  
\end{equation}

\begin{rem} \label{n12} \mbox{}
\begin{enu}
\item 
If $a_1= 1$ (resp., $a_2= 1$), then $r_2\lambda =\Lambda_2$ 
(resp., $r_1\lambda = -\Lambda_1$) is dominant (resp., antidominant). 
Hence, $\lambda$ does not satisfy the condition \eqref{eq31} in these cases. 

\item 
In the case that $a_1 = a_2 = 2$, i.e., $A$ is of type $A_1^{(1)}$, 
$\lambda=\Lambda_{1}-\Lambda_{2}$ is nothing but the level-zero fundamental weight, 
and all the results mentioned in this subsection have been known as a special case
(the easiest case) of \cite{o5} and \cite{o7}. 
However, it is worth mentioning that 
our proofs for the theorems and corollary below 
are still valid also in this affine case, and 
are slightly different from their proofs.
\end{enu}
\end{rem}

\begin{thm}[will be proved in \S\ref{subsec:prf1}] \label{n7}
For each $b \in \CB(\lambda)$, there exist $i_1,\,\dots,\,i_k \in I$ 
such that $b = \ti{f}_{i_k} \cdots \ti{f}_{i_1} u_{\lambda}$ or 
$b = \ti{e}_{i_k} \cdots \ti{e}_{i_1} u_{\lambda}$. 
In particular, the crystal graph of $\CB(\lambda)$ is connected.
\end{thm}
The next corollary follows immediately from Theorem~\ref{n7}.
\begin{cor} \label{n14}
For every $\mu \in P$, the subset $\CB(\lambda)_{\mu}$ of 
elements of weight $\mu$ in $\CB(\lambda)$ is a finite set. 
In particular, $\CB(\lambda)_{\lambda} = \{u_{\lambda}\}$.
\end{cor}

Recall from the comment after \eqref{eq30} 
that $u_\lambda$ is an extremal element of weight $\lambda$.

\begin{thm}[will be proved in \S\ref{subsec:prf2}] \label{n10} \mbox{}
\begin{enu}
\item Let $x, y \in W$. Then, $S_x u_\lambda = S_y u_\lambda$ 
if and only if $x\lambda = y\lambda$.

\item If $b \in \CB(\lambda)$ is extremal, 
then there exists  $w \in W$ such that $b = S_wu_\lambda$.
\end{enu}
\end{thm}

\begin{thm}[will be proved in \S\ref{subsec:prf3}] \label{n1124(1)}
The extremal weight module $V(\lambda)$ 
is irreducible. 
\end{thm}

\begin{thm}[will be proved in \S\ref{sec:prf4}] \label{20171126(4)}
There exists an isomorphism 
$\CB(\lambda) \stackrel{\sim}{\rightarrow} \BB(\lambda)$ 
of crystals that sends $u_\lambda$ to $\pi_\lambda$.
\end{thm}
%
%
\subsection{Proof of Theorem~\ref{n7}.}
\label{subsec:prf1}

\begin{lem} \label{n4}
Let $i \in I$ and $b \in \CB(\lambda)$ be such that $\ti{e}_i b \neq \bzero$.
If $b$ is of the form: $b = b_1 \otimes t_\lambda \otimes u_{-\infty}$ 
with $b_1 \neq u_\infty$, then 
$\ti{e}_i b = \ti{e}_i b_1 \otimes t_\lambda \otimes u_{-\infty}$.
\end{lem}

\begin{proof} 
Suppose, for a contradiction, that 
$\ti{e}_i b = b_1 \otimes t_\lambda \otimes \ti{e}_i u_{-\infty}$. 
By \eqref{eq26} and Lemma \ref{n11}, we see that 
\begin{equation} \label{eq08}
(\ti{e}_i b)^{\ast} = b_1^{\ast} \otimes t_{-\lambda-\wt(b_1)-\alpha_i} \otimes \ti{e}_i u_{-\infty}.
\end{equation}
Because the third tensor factor $\ti{e}_i u_{-\infty}$ above 
satisfies that $\ti{f}_{i}(\ti{e}_i u_{-\infty}) = u_{-\infty} \ne \bzero$ 
(or equivalently, $\vp_{i}(\ti{e}_i u_{-\infty}) = 1 \ne 0$), 
it follows from the tensor product rule of crystals that 
$\ti{f}_i(\ti{e}_i b)^\ast \neq \bzero$.
Since $\ti{e}_i b \in \CB(\lambda)$, it follows that 
$(\ti{e}_i b)^\ast$ is an extremal element of weight $-\lambda$ (see Remark~\ref{rem:wt}). 
Hence we see that $\ti{f}_1(\ti{e}_1 b)^\ast = \bzero$, which implies that $i = 2$, and 
$(\ti{e}_2 b)^{\ast} = b_1^{\ast} \otimes t_{-\lambda-\wt(b_1)-\alpha_2} \otimes \ti{e}_2 u_{-\infty}$.
Since $(\ti{e}_2 b)^{\ast}$ is an extremal element of weight $-\lambda$, 
we see that $\ti{e}_1^2(\ti{e}_2 b)^{\ast} = \ti{e}_1S_1(\ti{e}_2 b)^{\ast} = \bzero$ 
and $\ti{e}_2(\ti{e}_2 b)^{\ast} = \bzero$. 
By these equalities and the tensor product rule of crystals, 
we have $\ve_1(b_1^\ast) \leq \ve_1((\ti{e}_2 b)^{\ast}) = 1$ and 
$\ve_2(b_1^\ast) \leq \ve_2((\ti{e}_2 b)^\ast) = 0$. Thus we get $\ve_2(b_1^\ast) = 0$. 
Moreover, since $b_1 \neq u_\infty$ by assumption, we obtain $\ve_1(b_1^{\ast}) = 1$. 
Thus, $b_1^\ast$ is of the form : $b_1^\ast = \ti{f}_1 b_1'$ 
for some $b_1' \in \CB(\infty)$ such that $\ti{e}_1 b_1' = \bzero$. 
Because $\ti{e}_1^2(\ti{e}_2b)^\ast = \ti{e}_1S_1(\ti{e}_2b)^\ast = \bzero$ as seen above, 
we see by the tensor product rule of crystals that
\begin{equation*}
S_1(\ti{e}_2b)^{\ast} = \ti{e}_1(\ti{e}_2 b)^{\ast} = 
\ti{e}_1b_1^{\ast} \otimes t_{-\lambda-\wt(b_1)-\alpha_2} \otimes \ti{e}_2 u_{-\infty} = 
b_1' \otimes t_{-\lambda-\wt(b_1)-\alpha_2} \otimes \ti{e}_2 u_{-\infty};
\end{equation*}
note that $\ti{f}_2S_1(\ti{e}_2b)^{\ast} \neq \bzero$ 
by the tensor product rule of crystals. 
However, since $S_1(\ti{e}_2b)^{\ast}$ is an extremal element of weight $-r_1\lambda$, and 
since $\pair{-r_1\lambda}{\alpha_2^{\vee}} = -a_2 + 1 \leq 0$, 
we have $\ti{f}_2(S_1(\ti{e}_2b)^{\ast}) = \bzero$, which is a contradiction. 
Thus we have proved the lemma.
\end{proof}

The proof of the next lemma is similar to the proof of Lemma \ref{n4}.

\begin{lem}\label{n5}
Let $i \in I$ and $b \in \CB(\lambda)$ be such that 
$\ti{f}_i b \neq \bzero$. If $b$ is of the form: 
$b = u_\infty \otimes t_\lambda \otimes b_2$ with 
$b_2 \neq u_{-\infty}$, then 
$\ti{f}_i b = u_\infty \otimes t_\lambda \otimes \ti{f}_i b_2$.
\end{lem}

\begin{prop}\label{n6}
It holds that $\CB(\lambda) \subset 
(\CB(\infty) \otimes t_\lambda \otimes u_{-\infty}) \cup 
(u_\infty \otimes t_\lambda \otimes \CB(-\infty))$.
\end{prop}

\begin{proof}
By Lemmas \ref{n4} and \ref{n5}, the subset
\begin{equation*}
\CB(\lambda) \cap \bigl((\CB(\lambda) \otimes t_\lambda \otimes u_{-\infty}) \cup 
(u_\infty \otimes t_\lambda \otimes \CB(\lambda))\bigr) =: \CC
\end{equation*}
is a subcrystal of $\CB(\lambda)$. Therefore it suffices to show that 
every element $b \in \CB(\lambda)$ is connected to an element of $\CC$.
Write $b$ as $b = b_1 \otimes t_\lambda \otimes b_2$ 
with $b_1 \in \CB(\infty)$ and $b_2 \in \CB(-\infty)$. 
It is known that there exist $i_1,\,\dots,\,i_k \in I$ 
such that $\ti{f}_{i_k}^{\max} \cdots \ti{f}_{i_1}^{\max}b_2 = u_{-\infty}$. 
Then we see by the tensor product rule of crystals that 
$\ti{f}_{i_k}^{\max} \cdots \ti{f}_{i_1}^{\max}b = 
b_1' \otimes t_\lambda \otimes u_{-\infty}$ for some $b_1' \in \CB(\infty)$, 
which implies that $\ti{f}_{i_k}^{\max} \cdots \ti{f}_{i_1}^{\max} b \in \CC$. 
Thus we have proved the proposition.
\end{proof}

Here, for $\alpha = \sum_{i \in I} c_i\alpha_i \in Q:=\bigoplus_{i \in I}\BZ \alpha_{i}$, 
we set $|\alpha| := \sum_{i \in I} |c_i| \in \BZ_{\geq 0}$. 
Also, we recall that $-\wt(b_{1}) \in Q_{+}$ for all $b_{1} \in \CB(\infty)$ and 
$\wt(b_{2}) \in Q_{+}$ for all $b_{2} \in \CB(-\infty)$. 

\begin{proof}[Proof of Theorem~\ref{n7}]
By Proposition \ref{n6}, $b \in \CB(\lambda)$ is either of 
the following forms: $b = b_1 \otimes t_\lambda \otimes u_{-\infty}$ 
for some $b_1 \in \CB(\infty)$, or 
$b = u_\infty \otimes t_\lambda \otimes b_2$ 
for some $b_2 \in \CB(-\infty)$. 
We show by induction on $|\wt(b_1)|$ that 
if $b \in \CB(\lambda)$ is the form $b = b_1 \otimes t_\lambda \otimes u_{-\infty}$ 
for some $b_1 \in \CB(\infty)$, then $b = \ti{f}_{i_k} \cdots \ti{f}_{i_1} u_{\lambda}$ 
for some $i_1,\,\dots,\,i_k \in I$. 
If $|\wt(b_1)| = 0$, then the assertion is obvious 
since $b_1 = u_\infty$, and hence $b = u_\lambda$.
Assume that $|\wt(b_1)| \geq 1$. Since $b_1 \neq u_\infty$, 
there exists $i \in I$ such that $\ti{e}_i b_1 \neq \bzero$; 
we see by the tensor product rule of crystals that 
$\ti{e}_i b \neq \bzero$. Moreover, we deduce by Lemma \ref{n4} that
\begin{equation*}
\ti{e}_i b = \ti{e}_i (b_1 \otimes t_\lambda \otimes u_{-\lambda}) = 
\ti{e}_i b_1 \otimes t_\lambda \otimes u_{-\infty}.
\end{equation*}
Since $|\wt(\ti{e}_{i}b_1)| = k - 1$, it follows 
by the induction hypothesis that 
there exist $i_1,\,\dots,\,i_k \in I$ such that 
$\ti{e}_i b = \ti{f}_{i_k} \cdots \ti{f}_{i_1} u_\lambda$. 
Then we obtain $b = \ti{f}_{i}\ti{f}_{i_k} \cdots \ti{f}_{i_1} u_{\lambda}$.

Similarly, we show by induction on $|\wt(b_2)|$ that 
if $b \in \CB(\lambda)$ is the form $b = u_\infty \otimes t_\lambda \otimes b_2$ 
for some $b_2 \in \CB(-\infty)$, then 
$b = \ti{e}_{i_k} \cdots \ti{e}_{i_1} u_{\lambda}$ 
for some $i_1,\,\dots,\,i_k \in I$. Thus we have proved Theorem~\ref{n7}.
\end{proof}
%
%
\subsection{Proof of Theorem~\ref{n10}.}
\label{subsec:prf2}

First, we show part (1) of Theorem~\ref{n10}. 
The ``only if'' part is obvious. To show the ``if'' part, 
assume that $x \lambda = y \lambda$ for $x, y \in W$. 
Then the weight of the element $S_{x ^{-1}} S_y u_\lambda$ 
is equal to $x^{-1}y\lambda = \lambda$. Therefore, by Corollary~\ref{n14}, 
we obtain $S_{x ^{-1}} S_y u_\lambda = u_\lambda$, 
and hence $S_x u_\lambda = S_y u_\lambda$, as desired.

Next, we show part (2) of Theorem~\ref{n10}. 
Let $b \in \CB(\lambda)$ be an extremal element. 
By Proposition~\ref{n6}, $b$ is either of the following forms:
$b = b_1 \otimes t_\lambda \otimes u_{-\infty}$ 
for some $b_1 \in \CB(\infty)$ or  $b = u_\infty \otimes t_\lambda \otimes b_2$ 
for some $b_2 \in \CB(-\infty)$.
We show by induction on $|\wt(b_1)|$ that 
if an extremal element $b \in \CB(\lambda)$ is of the form: 
$b = b_1 \otimes t_\lambda \otimes u_{-\infty}$ for some $b_1 \in \CB(\infty)$, 
then there exists $w \in W$ such that $b = S_w u_\lambda$. 
If $|\wt(b_1)| = 0$, then the assertion is obvious 
since $b_1 = u_\infty$,  and hence $b = u_\lambda$. 
Assume that $|\wt(b_1)| > 0$. There exists $i \in I$ such that $\ti{e}_i b_1 \neq \bzero$;
notice that $\ti{e}_i b \neq \bzero$ by the tensor product rule of crystals.
Since $b$ is an extremal element by assumption, 
we see that $\vp_i(b) = 0$ and $\pair{\wt(b)}{\alpha_i^{\vee}} = -\ve_i(b) \leq -1$; 
we set $n := - \pair{\wt(b)}{\alpha_i^{\vee}} \geq 1$. We deduce that
\begin{align*}
S_i b 
& = \ti{e}_i^n (b_1 \otimes t_\lambda \otimes u_{-\infty}) 
  = \ti{e}_i^{n-1} (\ti{e}_i b_1 \otimes t_\lambda \otimes u_{-\infty}) \quad 
    \text{by Lemma~\ref{n4}} \\
& = (\ti{e}_i^{n} b_1) \otimes t_\lambda \otimes u_{-\infty} \quad 
    \text{by the tensor product rule of crystals}.
\end{align*}
Because $S_i b$ is also an extremal element in $\CB(\lambda)$, and 
$|\wt(\ti{e}_i^{n} b_1)| < |\wt(b_1)|$, 
it follows from the induction hypothesis that 
$S_i b = S_w u_\lambda$ for some $w \in W$. 
Thus we get $b = S_i S_w u_\infty = S_{r_i w}u_\lambda$. 

Similarly, we can show by induction on $|\wt(b_2)|$ that 
if an extremal element $b \in \CB(\lambda)$ is of the form: 
$b = u_\infty \otimes t_\lambda \otimes b_2$ for some $b_2 \in \CB(-\infty)$, 
then there exists $w \in W$ such that $b = S_w u_\lambda$. 
Thus we have proved Theorem~\ref{n10}.
%
%
\subsection{Proof of Theorem \ref{n1124(1)}.}
\label{subsec:prf3}

Let $U \subset V(\lambda)$ be a $U_{\q}(\Fg)$-submodule of $V(\lambda)$ not identical to $\{0\}$, 
and let $U=\bigoplus_{\mu \in P} U_{\mu}$ be the weight space decomposition of $U$. 
We see from Theorem~\ref{n7} that if $U_{\mu} \ne \{0\}$ for $\mu \in P$, 
then $\mu \in \lambda-Q_{+}$ or $\mu \in \lambda+Q_{+}$. Here we show that 
if $U_{\mu} \ne \{0\}$ for some $\mu \in \lambda-Q_{+}$, then $U=V(\lambda)$. 
Take $\mu \in \lambda-Q_{+}$ in such a way that 
$U_{\mu} \ne \{0\}$, and $|\lambda-\nu| \ge |\lambda-\mu|$ 
for all $\nu \in \lambda-Q_{+}$ such that $U_{\nu} \ne \{0\}$.
If $\mu = \lambda$, then it follows from Corollary~\ref{n14} that $U_{\mu}=U_{\lambda}$ 
contains the extremal weight vector $v_{\lambda} \in V(\lambda)$. 
Since $V(\lambda)=U_{\q}(\Fg)v_{\lambda}$, we obtain $U=V(\lambda)$, as desired. 
Suppose, for a contradiction, that $\mu \ne \lambda$. 
Let $u \in U_{\mu}$, $u \ne 0$. 
Then, by the minimality of $\mu$, we see that 
$E_{i}u = 0$ for all $i \in I$, where $E_{i}$, $i \in I$, is 
the Chevalley generator of $U_{\q}(\Fg)$ corresponding to $\alpha_{i}$. 
Recall from \S\ref{subsec:extmod} that $V(\lambda)$ 
has a global basis $\bigl\{ G(b) \mid b \in \CB(\lambda) \bigr\}$.  
Write the $u$ above as: 
\begin{equation*}
u = f_{1}(\q)G(b_{1}) + f_{2}(\q)G(b_{2}) + \cdots + f_{m}(\q)G(b_{m}), 
\end{equation*}
where $b_{1},\,\dots,\,b_{m} \in \CB(\lambda)_{\mu}$ and 
$f_{1}(\q),\,\dots,\,f_{m}(\q) \in \BC(\q) \setminus \{0\}$. 
We see that there exists $g(\q) \in \BC(\q)$ such that 
$g_{k}(\q):= f_{k}(\q)g(\q) \in \BA_0$ for all $1 \le k \le m$, and 
there exists $1 \le k \le m$ such that $c_{k} := g_{k}(0) \neq 0$;
we may assume that $c_{1} \ne 0$. Then, $u':=g(\q)u$ is an element of 
the crystal lattice $\CL(\lambda)$ of $V(\lambda)$, and 
\begin{equation*}
u' + \q\CL(\lambda) = 
c_{1}b_{1} + c_{2}b_{2} + \cdots + c_{m}b_{m} 
\quad \text{in $\CL(\lambda)/\q\CL(\lambda)$}.
\end{equation*}
Since $\wt(b_{1}) = \mu \in \lambda - Q_{+}$, 
we deduce from Theorem~\ref{n7} that 
there exists $i \in I$ such that $\ti{e}_{i}b_{1} \ne 0$; 
we may assume without loss of generality 
that $\ti{e}_{i}b_{k} \ne 0$ for all $1 \le k \le n$, and 
$\ti{e}_{i}b_{k} = 0$ for all $n+1 \le k \le m$. 
Since $E_{i}u'=0$ in $V(\lambda)$, it follows from the definition of 
the raising Kashiwara operator $\ti{e}_{i}$ on $V(\lambda)$ that 
$\ti{e}_{i}u' = 0$ in $V(\lambda)$. 
Therefore, we have $\ti{e}_{i}(u'+\q\CL(\lambda)) = 0$ in 
$\CL(\lambda)/\q\CL(\lambda)$, and hence
\begin{equation*}
c_{1}\ti{e}_{i}b_{1} + c_{2}\ti{e}_{i}b_{2} + \cdots + c_{n}\ti{e}_{i}b_{n} = 0 \quad 
\text{in $\CL(\lambda)/\q\CL(\lambda)$}.
\end{equation*}
However, because $\ti{e}_{i}b_{k}$, $1 \le k \le n$, 
are linearly independent, this equality contradicts $c_{1} \ne 0$. 

Similarly, we can show that 
if $U_{\mu} \ne \{0\}$ for some $\mu \in \lambda+Q_{+}$, 
then $U=V(\lambda)$. Thus we have proved Theorem~\ref{n1124(1)}.
%
%
\section{Proof of the isomorphism theorem.}
\label{sec:prf4}

In this section (except for the proof of Theorem~\ref{20171126(4)}), 
we employ the notation and setting in \S\ref{subsec:km}. 
Namely, we assume that $\Fg$ is an arbitrary symmetrizable Kac-Moody algebra. 
For $\mu \in P$, we denote by $\CB_{0}(\mu)$ 
(resp., $\BB_{0}(\mu)$) the connected component of 
$\CB(\mu)$ (resp., $\BB(\mu)$) containing $u_{\mu}$ 
(resp., $\pi_{\mu}$). 

\begin{thm} \label{thm:isom2}
Let $\mu \in P$ be an integral weight. 
If $\CB_{0}(\mu)_{\mu} = \bigl\{ u_{\mu} \bigr\}$ and 
$\BB_{0}(\mu)_{\mu} = \bigl\{ \pi_{\mu} \bigr\}$, then 
there exists an isomorphism $\CB_{0}(\mu) \stackrel{\sim}{\rightarrow} 
\BB_{0}(\mu)$ of crystals that sends $u_{\mu}$ to $\pi_{\mu}$. 
\end{thm}

Here, let us prove Theorem~\ref{20171126(4)}, 
assuming that Theorem~\ref{thm:isom2} is true. 

\begin{proof}[Proof of Theorem~\ref{20171126(4)}]
Recall that the crystals $\CB(\lambda)$ and $\BB(\lambda)$ are 
connected by Theorem~\ref{n7} and \cite[Theorem~3.2.1]{o1}, respectively.  
Also, recall from Corollary \ref{n14} that 
$\CB(\lambda)_{\lambda} = \bigl\{ u_{\lambda} \bigr\}$. 
Further, we see from \cite[Theorem 4.2.1]{o1} that 
$\BB(\lambda)_{\lambda} = \bigl\{ \pi_{\lambda} \bigr\}$. 
Therefore, Theorem~\ref{20171126(4)} follows immediately 
from Theorem~\ref{thm:isom2}. 
\end{proof}

Thus it remains to prove Theorem~\ref{thm:isom2}.
By using ``\emph{similarities}'' of 
the crystals $\BB(\mu)$ and $\CB_{0}(\mu)$ 
(see Propositions \ref{prop:s1} and \ref{prop:s2} below, respectively), 
we can prove this theorem in exactly the same way as 
\cite[Theorem~4.1]{o13} and \cite[Theorem~5.1]{o7} (see \S\ref{subsec:prf4} below). 
We remark that the assumptions $\CB_{0}(\mu)_{\mu} = \bigl\{ u_{\mu} \bigr\}$ and 
$\BB_{0}(\mu)_{\mu} = \bigl\{ \pi_{\mu} \bigr\}$ are used only in Lemma~\ref{lem:equal}. 
%
%
\subsection{Similarity of the crystal $\CB(\mu)$.}
\label{subsec:sim1}

\begin{prop} \label{prop:s1}
Let $\mu \in P$ be an arbitrary integral weight, 
and $m \in \BZ_{\ge 1}$. There exists an injective map 
$\Sigma_{m}:\BB(\mu) \hookrightarrow \BB(\mu)^{\otimes m}$ 
$= \BB(\mu) \otimes \cdots \otimes \BB(\mu)$ ($m$ times) 
satisfying the following conditions:
\begin{equation} \label{eq:n1}
\begin{split}
& \Sigma_{m}(\pi_{\mu}) = \pi_{\mu}^{\otimes m},  \\
& \wt (\Sigma_{m}(\pi)) = m \wt (\pi) 
  \quad \text{\rm for $\pi \in \BB(\mu)$}, \\
& \Sigma_{m}(\ti{e}_{i}\pi) = \ti{e}_{i}^{m}\Sigma_{m}(\pi), \ 
  \Sigma_{m}(\ti{f}_{i}\pi) = \ti{f}_{i}^{m}\Sigma_{m}(\pi)
  \quad \text{\rm for $\pi \in \BB(\mu)$ and $i \in I$}, \\
& \ve_{i} (\Sigma_{m}(\pi)) = m \ve_{i} (\pi), \ 
  \vp_{i} (\Sigma_{m}(\pi)) = m \vp_{i} (\pi)
  \quad \text{\rm for $\pi \in \BB(\mu)$ and $i \in I$},
\end{split}
\end{equation}
where we understand $\Sigma_{m}(\bzero)=\bzero$. 
Moreover, for every $m,n \in \BZ_{\ge 1}$, 
the following diagram commutes:
\begin{equation} \label{eq:mn1}
\begin{split}
\xymatrix{
\BB(\mu) 
\ar[r]^-{\Sigma_{m}}
\ar[d]_{\Sigma_{mn}}
& \BB(\mu)^{\otimes m} 
\ar[ld]^{\Sigma_{n}^{\otimes m}} \\
\BB(\mu)^{\otimes mn}
}
\end{split}
\end{equation}
\end{prop}

We can obtain the map 
$\Sigma_{m}:\BB(\mu) \hookrightarrow \BB(\mu)^{\otimes m}$ above 
in exactly the same way as the map $\sigma_{m,\mu}$ with the notation in \cite[\S4.1]{o7}. 
Let us explain the outline. 
First, for $\pi \in \BB(\mu)$, we define 
$m\pi : [0, 1] \rightarrow \BR \otimes_{\BZ} P$ by 
$(m\pi)(t):=m\pi(t)$ for $t \in [0,1]$.
We deduce from the definition of LS paths that 
$m\pi \in \BB(m\mu)$; indeed, we see that 
if $\pi=(\nu_{1},\dots,\nu_{s}\,;\,\sigma_{0},
\sigma_{1},\dots,\sigma_{s}) \in \BB(\mu)$, then 
$m\pi=(m\nu_{1},\dots,m\nu_{s}\,;\,\sigma_{0},
\sigma_{1},\dots,\sigma_{s}) \in \BB(m\mu)$. 
Furthermore, we know from \cite[Lemma~2.4]{o10} that 
the map $\Theta_{m}:\BB(\mu) \rightarrow \BB(m\mu)$, $\pi \mapsto m\pi$, 
satisfies
\begin{equation} \label{eq:n1a}
\begin{split}
& \Theta_{m}(\pi_{\mu}) = \pi_{m\mu}, \\
& \wt (\Theta_{m}(\pi)) = m \wt (\pi) 
  \quad \text{for $\pi \in \BB(\mu)$}, \\
& \Theta_{m}(\ti{e}_{i}\pi) = \ti{e}_{i}^{m}\Theta_{m}(\pi), \ 
  \Theta_{m}(\ti{f}_{i}\pi) = \ti{f}_{i}^{m}\Theta_{m}(\pi)
  \quad \text{for $\pi \in \BB(\mu)$ and $i \in I$}, \\
& \ve_{i} (\Theta_{m}(\pi)) = m \ve_{i} (\pi), \ 
  \vp_{i} (\Theta_{m}(\pi)) = m \vp_{i} (\pi)
  \quad \text{for $\pi \in \BB(\mu)$ and $i \in I$},
\end{split}
\end{equation}
where we understand $\Theta_{m}(\bzero)=\bzero$. 

Next, for $\pi_{1},\pi_{2},\dots,\pi_{m} \in \BB(\mu)$, 
we define a concatenation 
$\pi_{1} \ast \pi_{2} \ast \cdots \ast \pi_{m}$ of them by:
\begin{equation*}
\begin{split}
& (\pi_{1} \ast \pi_{2} \ast \cdots \ast \pi_{m})(t)=
  \sum_{l=1}^{k-1} 
  \pi_{l}(1)+ \pi_{k}(mt-k+1) \nonumber \\[-1.5mm]
& \hspace{60mm} \text{for \ } 
\frac{k-1}{m} \le t \le \frac{k}{m}, \ 
1 \le k \le m,
\end{split}
\end{equation*}
and set
\begin{equation*}
\BB(\mu)^{\ast m}=
 \underbrace{\BB(\mu) \ast \cdots \ast \BB(\mu)}_{\text{$m$ times}}:=
\bigl\{ \pi_{1} \ast \cdots \ast \pi_{m} \mid 
\pi_{k} \in \BB(\mu),\,1 \le k \le m \bigr\}. 
\end{equation*}
We endow $\BB(\mu)^{\ast m}$ with a crystal structure as follows. 
Let $\pi = \pi_{1} \ast \pi_{2} \ast \cdots \ast \pi_{m} \in \BB(\mu)^{\ast m}$. 
First we define $\wt(\pi):=\pi(1) = \pi_{1}(1) + \cdots + \pi_{m}(1)$; 
notice that $\pi(1) \in P$ since $\pi_{k}(1) \in P$ for all $1 \le k \le m$.
Next, for $i \in I$, we define $\ti{e}_{i}\pi$ and $\ti{f}_{i}\pi$ 
in exactly the same way as for elements in $\BB(\mu)$ (see \S\ref{subsec:LS}); 
notice that the condition \eqref{eq:eq106} holds 
for every element in $\BB(\mu)^{\ast m}$. We deduce that 
if $\ti{e}_{i}\pi \ne \bzero$, then 
$\ti{e}_{i}\pi = \pi_{1} \ast \cdots \ast \ti{e}_{i}\pi_{k} \ast \cdots \ast \pi_{m}$ 
for some $1 \le k \le m$; the same holds also for $\ti{f}_{i}$. 
Therefore the set $\BB(\mu)^{\ast m} \cup \{\bzero\}$ is stable under 
the action of $\ti{e}_{i}$ and $\ti{f}_{i}$ for $i \in I$. 
Finally, for $i \in I$, we set $\ve_{i}(\pi) := 
\max \{n \ge 0 \mid \ti{e}_{i}^{n}\pi \neq \bzero \}$ and 
$\vp_{i}(\pi) := \max \{n \ge 0 \mid \ti{f}_{i}^{n}\pi \neq \bzero \}$. 
We know the following proposition from \cite[\S2]{o10}. 

\begin{prop} \label{prop:cat}
Let $\mu \in P$ be an arbitrary integral weight, and $m \in \BZ_{\geq 1}$.
The set $\BB(\mu)^{\ast m}$ together with the maps 
$\wt:\BB(\mu)^{\ast m} \rightarrow P$, 
$\ti{e}_{i},\,\ti{f}_{i}:\BB(\mu)^{\ast m} \cup \{\bzero\} 
\rightarrow \BB(\mu)^{\ast m} \cup \{\bzero\}$, $i \in I$, and 
$\ve_{i},\,\vp_{i}:\BB(\mu)^{\ast m} \rightarrow \BZ_{\ge 0}$, $i \in I$, 
becomes a crystal. Moreover, the map 
\begin{equation} \label{eq:cat}
\BB(\mu)^{\ast m} \rightarrow \BB(\mu)^{\otimes m}, \ 
\pi_{1} \ast \cdots \ast \pi_{m} \mapsto \pi_{1} \otimes \cdots \otimes \pi_{m}, 
\end{equation}
is an isomorphism of crystals. 
\end{prop}

Let $\pi \in \BB(m\mu)$. If we set 
\begin{equation*}
\pi_{k}(t):=
  \pi\left(\frac{1}{m}t + \frac{k-1}{m}\right) - 
  \pi\left( \frac{k-1}{m} \right) \quad \text{for $t \in [0,1]$}
\end{equation*}
for $1 \le k \le m$, then $\pi_{k} \in \BB(\mu)$ for all $1 \le k \le m$, and 
$\pi = \pi_{1} \ast \cdots \ast \pi_{m}$ (see Example~\ref{20171126(5)} below), 
which implies that $\BB(m\mu)$ is contained in $\BB(\mu)^{\ast m}$, 
and hence is a subcrystal of $\BB(\mu)^{\ast m} \cong \BB(\mu)^{\otimes m}$.
\begin{ex} \label{20171126(5)}
Assume that $m=2$, and $\pi = (\nu_{1},\dots,\nu_{s};
\sigma_{0},\sigma_{1},\dots,\sigma_{s}) \in \BB(2\mu)$; 
for simplicity, assume that $\sigma_{u-1} < 1/2 < \sigma_{u}$ 
for some $1 \le u \le s$. Then, 
\begin{align*}
\pi_{1} & = \left( \frac{1}{2}\nu_{1},\dots,\frac{1}{2}\nu_{u};
2\sigma_{0},2\sigma_{1},\dots,2\sigma_{u-1},1 \right) \in \BB(\mu),  \\[3mm]
\pi_{2} & = \left(\frac{1}{2}\nu_{u},\dots,\frac{1}{2}\nu_{s};
0,2\sigma_{u}-1,\dots,2\sigma_{s}-1\right) \in \BB(\mu).
\end{align*}
\end{ex}

\begin{proof}[Proof of Proposition~\ref{prop:s1}]
We define $\Sigma_{m}:\BB(\mu) \hookrightarrow \BB(\mu)^{\otimes m}$ to be 
the composition of the following maps: 
\begin{equation*}
\BB(\mu) 
\xrightarrow{\Theta_{m}} \BB(m\mu) 
\xrightarrow{\subset} \BB(\mu)^{\ast m} 
\xrightarrow{\sim} \BB(\mu)^{\otimes m}.
\end{equation*}
Then we can easily check by the definitions of these maps and \eqref{eq:n1a} that 
$\Sigma_{m}$ satisfies the conditions \eqref{eq:n1}, and 
the diagram \eqref{eq:mn1} commutes.
\end{proof}
%
%
\subsection{Similarity of the crystal $\CB_0(\mu)$.}
\label{subsec:sim2}

\begin{prop} \label{prop:s2}
Let $\mu \in P$ be an arbitrary integral weight, 
and $m \in \BZ_{\ge 1}$. There exists a map 
$\Sigma_{m}:\CB_0(\mu) \hookrightarrow \CB_0(\mu)^{\otimes m} = 
\CB_0(\mu) \otimes \cdots \otimes \CB_0(\mu)$ ($m$ times) 
satisfying the following conditions\,{\rm:}
\begin{equation} \label{eq:n2}
\begin{split}
& \Sigma_{m}(u_{\mu}) = u_{\mu}^{\otimes m}, \\
& \wt (\Sigma_{m}(b)) = m \wt (b) 
  \quad \text{\rm for $b \in \CB_0(\mu)$}, \\
& \Sigma_{m}(\ti{e}_{i}b) = \ti{e}_{i}^{m}\Sigma_{m}(b), \ 
  \Sigma_{m}(\ti{f}_{i}b) = \ti{f}_{i}^{m}\Sigma_{m}(b)
  \quad \text{\rm for $b \in \CB_0(\mu)$ and $i \in I$}, \\
& \ve_{i} (\Sigma_{m}(b)) = m \ve_{i} (b), \ 
  \vp_{i} (\Sigma_{m}(b)) = m \vp_{i} (b)
  \quad \text{\rm for $b \in \CB_0(\mu)$ and $i \in I$},
\end{split}
\end{equation}
where we understand $\Sigma_{m}(\bzero)=\bzero$. Moreover, 
for every $m,n \in \BZ_{\ge 1}$, 
the following diagram commutes: 
\begin{equation} \label{eq:mn2}
\begin{split}
\xymatrix{
\CB_0(\mu) 
\ar[r]^-{\Sigma_{m}}
\ar[d]_{\Sigma_{mn}}
& \CB_0(\mu)^{\otimes m} 
\ar[ld]^{\Sigma_{n}^{\otimes m}} \\
\CB_0(\mu)^{\otimes mn}
}
\end{split}
\end{equation}
\end{prop}

We know the following proposition from \cite[Theorem 3.7]{o7}. 

\begin{prop} \label{prop:n2}
Let $\mu \in P$ be an arbitrary integral weight, 
and $m \in \BZ_{\ge 1}$. There exists an injective map 
$\Theta_{m}:\CB(\mu) \hookrightarrow \CB(m\mu)$
satisfying the following conditions:
\begin{equation} \label{eq:n2a}
\begin{split}
& \Theta_{m}(u_{\mu}) = u_{m\mu},  \\
& \wt (\Theta_{m}(b)) = m \wt (b) 
  \quad \text{\rm for $b \in \CB(\mu)$}, \\
& \Theta_{m}(\ti{e}_{i}b) = \ti{e}_{i}^{m}\Theta_{m}(b), \ 
  \Theta_{m}(\ti{f}_{i}b) = \ti{f}_{i}^{m}\Theta_{m}(b)
  \quad \text{\rm for $b \in \CB(\mu)$ and $i \in I$}, \\
& \ve_{i} (\Theta_{m}(b)) = m \ve_{i} (b), \ 
  \vp_{i} (\Theta_{m}(b)) = m \vp_{i} (b)
  \quad \text{\rm for $b \in \CB(\mu)$ and $i \in I$},
\end{split}
\end{equation}
where we understand $\Theta_{m}(\bzero)=\bzero$.
\end{prop}

\begin{rem} \label{rem:n2} \mbox{}
\begin{enu}
\item By \eqref{eq:n2a},  
the image $\Theta_{m}(\CB_0(\mu))$ of $\CB_0(\mu)$ under the map 
$\Theta_{m}:\CB(\mu) \rightarrow \CB(m\mu)$ is contained in $\CB_0(m\mu)$. 

\item Let $b \in \CB(\mu)$. 
Let $i \in I$, and let $\ti{x}_{i}$ be either $\ti{e}_{i}$ or $\ti{f}_{i}$. 
We see by \eqref{eq:n2a} that 
if $\ti{x}_{i}^{m}\Theta_{m}(b) \ne 0$, then $\ti{x}_{i}b \ne 0$\,{\rm;}
note that $\ti{x}_{i}b \in \CB(\mu)$, and 
$\ti{x}_{i}^{m}\Theta_{m}(b) = \Theta_{m}(\ti{x}_{i}b) \in \Theta_{m}(\CB(\mu))$. 
\end{enu}
\end{rem}

For convenience, we set $\ti{e}_{i}^{-1}:=\ti{f}_{i}$ and 
$\ti{f}_{i}^{-1}:=\ti{e}_{i}$ for $i \in I$. Also, 
for a monomial $X=\ti{x}_{i_s} \cdots \ti{x}_{i_1}$ in Kashiwara operators, 
where $\ti{x}_{i_u}$ is either $\ti{e}_{i_u}$ or $\ti{f}_{i_u}$ for each $1 \le u \le s$, 
we set $X^{(m)}:=\ti{x}_{i_s}^{m} \cdots \ti{x}_{i_1}^{m}$ and 
$X^{-1}:=\ti{x}_{i_1}^{-1} \cdots \ti{x}_{i_s}^{-1}$. 

\begin{lem} \label{lem:n3}
Let $\mu \in P$ be an arbitrary integral weight, and $m \in \BZ_{\ge 1}$.
There exists a strict morphism 
$\Psi_{m}:\CB_{0}(m\mu) \rightarrow \CB_0(\mu)^{\otimes m}$ 
of crystals (in the sense of \cite[\S7.6]{o2})
that sends $u_{m\mu}$ to $u_{\mu}^{\otimes m}$.
\end{lem}

\begin{proof}
Recall that the crystal lattice $\CL(m\mu)$ of $V(m\mu)$ is 
a free $\BA_{0}$-module which is stable under 
the Kashiwara operators on $V(m\mu)$ and 
contains the extremal weight vector $v_{m\mu} \in V(m\mu)$. 
For each element $b \in \CB_{0}(m\mu)$, 
we fix a monomial $X_{b}$ in Kashiwara operators such that $b=X_{b}u_{m\mu}$;
notice that $X_b v_{m\mu} \in \CL(\mu)$, and in $\CL(m\mu)/\q\CL(m\mu)$, 
\begin{equation} \label{eq:psi1}
X_{b} v_{m\mu} + \q\CL(m\mu) = 
X_{b} (v_{m\mu} + \q\CL(m\mu)) = 
X_{b}u_{m\mu} = b 
\end{equation}
for every $b \in \CB_{0}(m\mu)$. 
Define $\CL_{0}(m\mu)$ to be the $\BA_{0}$-submodule of $\CL(m\mu)$ 
generated by $\bigl\{ X_{b} v_{m\mu} \mid b \in \CB_{0}(m\mu) \bigr\} 
\subset \CL(m\mu)$. 
Recall that the global basis $\bigl\{G(b) \mid b \in \CB(m\mu) \bigr\}$ of $V(m\mu)$ 
is also an $\BA_{0}$-basis of $\CL(m\mu)$. 
Since $G(b) + \q\CL(\mu) = b$ in $\CL(m\mu)/\q\CL(m\mu)$, 
we see from \eqref{eq:psi1} that for each $b \in \CB_{0}(m\mu)$, 
\begin{equation} \label{eq:psi2}
X_{b} v_{m\mu} \in (1 + \q\BA_{0}) G(b) + 
 \sum_{b' \in \CB(m\mu),\,b'\ne b} \q\BA_{0} G(b').
\end{equation}
Here we claim that
\begin{equation} \label{eq:psi3}
\CL_{0}(m\mu) \cap \q\CL(m\mu) = \q\CL_{0}(m\mu). 
\end{equation}
The inclusion $\supset$ is obvious. 
Let us show the opposite inclusion $\subset$. 
Assume that 
\begin{equation*}
v = \sum_{b \in \CB_{0}(m\mu)} c_{b}(\q) X_{b} v_{m\mu} \in \CL_{0}(m\mu), \quad 
\text{with $c_{b}(\q) \in \BA_{0}$}, 
\end{equation*}
is contained also in $\q\CL(m\mu)$. 
By \eqref{eq:psi1} and \eqref{eq:psi2}, 
we see that in $\CL(m\mu)/\q\CL(m\mu)$, 
\begin{equation*}
0 + \q\CL(m\mu) = v + \q\CL(m\mu) = 
\sum_{b \in \CB_{0}(m\mu)} c_{b}(0) b. 
\end{equation*}
Therefore, $c_{b}(0) = 0$ for all $b \in \CB_{0}(m\mu)$, 
and hence $c_{b}(\q) \in \q\BA_{0}$ for all $b \in \CB_{0}(m\mu)$. 
Thus we obtain $v \in \q\CL_{0}(m\mu)$, as desired. 
Therefore, as $\BC$-vector spaces,
\begin{equation}
\CL_{0}(m\mu)/\q\CL_{0}(m\mu)  = 
\CL_{0}(m\mu)/(\CL_{0}(m\mu) \cap \q\CL(m\mu)) \subset 
\CL(m\mu)/\q\CL(m\mu). 
\end{equation}

Consider the tensor product $U_{\q}(\Fg)$-module 
$V(\mu)^{\otimes m}=V(\mu) \otimes \cdots \otimes V(\mu)$ ($m$ times). 
We deduce that $v_{\mu}^{\otimes m}=v_{\mu} \otimes \cdots \otimes v_{\mu} 
\in V(\mu)^{\otimes m}$ is an extremal weight vector of weight $m\mu$. 
Hence, by the universality of extremal weight modules, there exists a $U_{\q}(\Fg)$-module
homomorphism $\Phi_{m}:V(m\mu) \rightarrow V(\mu)^{\otimes m}$
that sends $v_{m\mu}$ to $v_{\mu}^{\otimes m}$; 
note that this map commutes with the Kashiwara operators on $V(m\mu)$ and 
$V(\mu)^{\otimes m}$. 
Recall that $V(\mu)^{\otimes m}$ has 
a crystal basis $(\CL(\mu)^{\otimes m}, \CB(\mu)^{\otimes m})$. 
Because $v_{\mu}^{\otimes m} \in \CL(\mu)^{\otimes m}$, 
we see that
the restriction of $\Phi_{m}:V(m\mu) \rightarrow V(\mu)^{\otimes m}$ 
to $\CL_{0}(m\mu) \subset V(m\mu)$ gives an $\BA_{0}$-module homomorphism 
from $\CL_{0}(m\mu)$ to $\CL(\mu)^{\otimes m}$. 
By \eqref{eq:psi3}, we obtain the induced $\BC$-linear map
\begin{equation}
\begin{split}
& \Psi_{m}:\CL_{0}(m\mu)/(\CL_{0}(m\mu) \cap \q\CL(m\mu)) 
  =\CL_{0}(m\mu)/\q\CL_{0}(m\mu) \rightarrow 
 \CL(\mu)^{\otimes m}/\q\CL(\mu)^{\otimes m}.
\end{split}
\end{equation}
Because $\Psi_{m}(u_{m\mu}) = u_{\mu}^{\otimes m}$ by the definition of $\Psi_{m}$, 
and because the map $\Psi_{m}$ commutes with the Kashiwara operators, 
we deduce that the image of 
$\CB_{0}(m\mu) \subset 
\CL_{0}(m\mu)/(\CL_{0}(m\mu) \cap \q\CL(m\mu))$ 
under the map $\Psi_{m}$ above is included in 
$\CB_0(\mu)^{\otimes m} \cup \{0\}$. 
Suppose that $\Psi_{m}(b) = 0$ for some $b \in \CB_{0}(m\mu)$. 
Write $b$ as: $b=Xu_{m\mu}$ for some monomial $X$ in Kashiwara operators. 
We have 
\begin{equation*}
0 = X^{-1} 0 
 = X^{-1} \Psi_{m}(b) 
 = \Psi_{m}(X^{-1} b )  = \Psi_{m}(X^{-1} X u_{m\mu} )
 = \Psi_{m}(u_{m\mu})=u_{\mu}^{\otimes m},
\end{equation*}
which is a contradiction. Thus we get 
$\Psi_{m}(\CB_{0}(m\mu)) \subset \CB_0(\mu)^{\otimes m}$. 
It is obvious that $\Psi_{m}:\CB_{0}(m\mu) \rightarrow \CB_0(\mu)^{\otimes m}$ 
is a strict morphism of crystals. Thus we have proved Lemma~\ref{lem:n3}. 
\end{proof}

\begin{proof}[Proof of Proposition~\ref{prop:s2}]
We define $\Sigma_{m}:\CB_0(\mu) \hookrightarrow \CB_0(\mu)^{\otimes m}$ to be 
the composition of the following maps (see also Remark \ref{rem:n2} (1)): 
\begin{equation*}
\CB_0(\mu) 
\xrightarrow{\Theta_{m}} \CB_0(m\mu) 
\xrightarrow{\Psi_{m}} \CB_0(\mu)^{\otimes m}.
\end{equation*}
It can be easily checked that the map $\Sigma_{m}=\Psi_{m} \circ \Theta_{m}$ above 
satisfies the conditions \eqref{eq:n2}. 
Also, we can show in exactly the same way as \cite[Proposition 8.3.2 (2)]{KaF} and 
\cite[Lemma 3.10]{o7} that the diagram \eqref{eq:mn2} commutes. 
Thus we have proved Proposition~\ref{prop:s2}. 
\end{proof}

\begin{cor} \label{20171126(6)}
Let $\mu \in P$, and $m \in \BZ_{\ge 1}$. 
If $\CB_0(\mu)_\mu = \{u_\mu\}$, then 
the map $\sum_m : \CB_0(\mu) \rightarrow \CB_0(\mu)^{\otimes m}$ is injective.
\end{cor}

\begin{proof}
Assume that $\Sigma_{m}(b_{1})=\Sigma_{m}(b_{2})$ for some $b_{1},b_{2} \in \CB(\mu)$. 
Write $b_{2}$ as: $b_{2}=X_{2}u_{\mu}$ 
for some monomial $X_{2}$ in Kashiwara operators. 
Then we have $\Sigma_{m}(b_{2})=X_{2}^{(m)}u_{\mu}^{\otimes m}$, and 
\begin{align*}
\Psi_{m}((X_{2}^{(m)})^{-1}\Theta_{m}(b_{1})) & = 
(X_{2}^{(m)})^{-1}\Psi_{m}(\Theta_{m}(b_{1}))   = 
(X_{2}^{(m)})^{-1}\Sigma_{m}(b_{1}) \\ 
& = (X_{2}^{(m)})^{-1} \Sigma_{m}(b_{2})
  = u_{\mu}^{\otimes m};
\end{align*}
in particular, $\Psi_{m}((X_{2}^{(m)})^{-1}\Theta_{m}(b_{1})) \ne 0$, and hence 
$(X_{2}^{(m)})^{-1}\Theta_{m}(b_{1}) \ne 0$. 
By using  Remark~\ref{rem:n2}\,(2) repeatedly, 
we see that $X_{2}^{-1}b_{1} \ne \bzero$. 
Since the weight of $(X_{2}^{(m)})^{-1}\Theta_{m}(b_{1})=\Theta_{m}(X_{2}^{-1}b_{1})$ 
is equal to $m\mu$ by the computation above, 
it follows from \eqref{eq:n2a} that the weight of $X_{2}^{-1}b_{1}$ is equal to $\mu$. 
By assumption, we get $X_{2}^{-1}b_{1} = u_{\mu}$, 
and hence $b_{1}=X_{2}u_{\mu}=b_{2}$, as desired. 
\end{proof}
%
%
\subsection{Proof of Theorem~\ref{thm:isom2}.}
\label{subsec:prf4}

As mentioned at the beginning of this section,
the proof of Theorem~\ref{thm:isom2} is similar to those of 
\cite[Theorem 4.1]{o13} and \cite[Theorem 5.1]{o7}. 
So, we give only a sketch of the proof. 
The following two lemmas can be shown in the same manner as 
\cite[Proposition 3.12]{o7} (see also \cite[page 181]{o13} and 
\cite[Proposition 8.3.2 (3)]{KaF}). 

\begin{lem}[{cf. \cite[Proposition 4.4]{o7}}] \label{lem:ext1}
Let $\mu \in P$ be an arbitrary integral weight.
Let $\pi \in \BB_0(\mu)$, and write it as: 
$\pi=\ti{x}_{i_{p}} \cdots \ti{x}_{i_{1}}\pi_{\mu}$, 
where $\ti{x}_{i_q}$ is either $\ti{e}_{i_q}$ or $\ti{f}_{i_q}$ for each $1 \le q \le p$. 
Set $\pi_{q}:=\ti{x}_{i_q} \cdots \ti{x}_{i_1}\pi_{\mu}$
for $0 \le q \le p$. There exists $m \in \BZ_{\ge 1}$ 
(having sufficiently many divisors) such that 
the elements $\Sigma_{m}(\pi_{q}) \in \BB_0(\mu)^{\otimes m}$, 
$0 \le q \le p$, are of the forms: 
\begin{equation}
\Sigma_{m}(\pi_{q}) = 
S_{w_{q,1}}\pi_{\mu} \otimes 
S_{w_{q,2}}\pi_{\mu} \otimes \cdots \otimes 
S_{w_{q,m}}\pi_{\mu}
\end{equation}
for some $w_{q,1},\,w_{q,2},\,\dots,\,w_{q,m} \in W$.
\end{lem}

\begin{lem}[{cf. \cite[Proposition 3.2]{o7}}] \label{lem:ext2}
Let $\mu \in P$ be an arbitrary integral weight.
Let $b \in \CB_0(\mu)$, and write it as: 
$b=\ti{x}_{i_{p}} \cdots \ti{x}_{i_{1}}u_{\mu}$, 
where $\ti{x}_{i_q}$ is either $\ti{e}_{i_q}$ or $\ti{f}_{i_q}$ for each $1 \le q \le p$. 
Set $b_{q}:=\ti{x}_{i_{q}} \cdots \ti{x}_{i_{1}}u_{\mu}$
for $0 \le q \le p$. There exists $m \in \BZ_{\ge 1}$ 
(having sufficiently many divisors) such that 
the elements $\Sigma_{m}(b_{q}) \in \CB_0(\mu)^{\otimes m}$, 
$0 \le q \le p$, are of the forms: 
\begin{equation}
\Sigma_{m}(b_{q}) = 
S_{z_{q,1}}u_{\mu} \otimes 
S_{z_{q,2}}u_{\mu} \otimes \cdots \otimes 
S_{z_{q,m}}u_{\mu}
\end{equation}
for some $z_{q,1},\,z_{q,2},\,\dots,\,z_{q,m} \in W$.
\end{lem}

Recall from Lemma~\ref{20141124(5)} and 
the comment after \eqref{eq30} that 
$\pi_{\mu} \in \BB_0(\mu)$ and $u_{\mu} \in \CB_0(\mu)$ 
are extremal elements of weight $\mu$.
By the definition of an extremal element, 
we see (cf. \cite[Lemma~4.5]{o7}) that for every $w \in W$ and $i \in I$, 
\begin{equation} \label{eq:ten}
\begin{split}
& \ve_{i}(S_{w}\pi_{\mu}) = 
  \ve_{i}(S_{w}u_{\mu}), \quad 
  \vp_{i}(S_{w}\pi_{\mu}) = 
  \vp_{i}(S_{w}u_{\mu}), \\
& \pair{\wt(S_{w}\pi_{\mu})}{\alpha_{i}^{\vee}} = 
  \pair{\wt(S_{w}u_{\mu})}{\alpha_{i}^{\vee}}. 
\end{split}
\end{equation}

\begin{lem}[{cf. \cite[Lemma 4.6]{o7}}] \label{lem:ext3} \mbox{}
\begin{enu}
\item Keep the notation and setting in Lemma~\ref{lem:ext1}. 
Then, for every $0 \le q \le p$, we have
$b_{q}:=\ti{x}_{i_{q}} \cdots \ti{x}_{i_{1}}u_{\mu} \ne \bzero$, and 
\begin{equation*}
\Sigma_{m}(b_{q}) = 
S_{w_{q,1}}u_{\mu} \otimes 
S_{w_{q,2}}u_{\mu} \otimes \cdots \otimes 
S_{w_{q,m}}u_{\mu}. 
\end{equation*}

\item Keep the notation and setting in Lemma \ref{lem:ext2}. 
Then, for every $0 \le q \le p$, we have
$\pi_{q}:=\ti{x}_{i_{q}} \cdots \ti{x}_{i_{1}}\pi_{\mu} \ne \bzero$, and 
\begin{equation*}
\Sigma_{m}(\pi_{q}) = 
S_{z_{q,1}}\pi_{\mu} \otimes 
S_{z_{q,2}}\pi_{\mu} \otimes \cdots \otimes 
S_{z_{p,m}}\pi_{\mu}. 
\end{equation*}
\end{enu}
\end{lem}

\begin{proof}
The proof is similar to that of \cite[Lemma 4.6]{o7}; 
since the tensor product rule of crystals is controlled by 
those values in \eqref{eq:ten} (see \cite[(4.3) and (4.4)]{o7}), 
we can show the assertions by induction on $0 \le q \le p$. 
\end{proof}

By Lemma \ref{lem:ext3}, we deduce 
(cf. \cite[(2) in the proof of Theorem 5.1]{o7}) that 
for a monomial $X$ in Kashiwara operators, 
\begin{equation}\label{eq:nonzero}
X\pi_{\mu} \ne \bzero \quad \text{if and only if} \quad Xu_{\mu} \ne \bzero.
\end{equation}

\begin{lem}[{cf. \cite[(1) in the proof of Theorem 5.1]{o7}}]  \label{lem:equal}
Assume that $\CB_0(\mu)_\mu = \{u_\mu\}$ and $\BB_0(\mu)_\mu = \{\pi_\mu\}$.
Let $X_{1},\,X_{2}$ be monomials in Kashiwara operators. Then, 
\begin{equation} \label{eq:equal}
X_{1}\pi_{\mu} = X_{2}\pi_{\mu}
\quad \text{\rm if and only if} \quad 
X_{1}u_{\mu} = X_{2}u_{\mu}. 
\end{equation}
\end{lem}

\begin{proof}
We give a proof only for the ``only if" part; the proof of the ``if" part is similar.
If $X_{1}\pi_{\mu} = X_{2}\pi_{\mu} = \bzero$, 
then the assertion follows immediately from \eqref{eq:nonzero}. 
Assume that $X_{1}\pi_{\mu} = X_{2}\pi_{\mu} \ne \bzero$. 
Then we have $X_{2}^{-1}X_{1}\pi_{\mu}=\pi_{\mu}$ 
(for the definition of $X_{2}^{-1}$, 
see the paragraph preceding Lemma~\ref{lem:n3}). 
Since $X_{2}^{-1}X_{1}\pi_{\mu} \ne \bzero$, 
it follows from \eqref{eq:nonzero} that 
$X_{2}^{-1}X_{1}u_{\mu} \ne \bzero$. 
Since the weight of $X_{2}^{-1}X_{1}u_{\mu}$ is equal to $\mu$, 
we see by the assumption $\CB_0(\mu)_\mu = \{u_\mu\}$ that $X_{2}^{-1}X_{1}u_{\mu}=u_{\mu}$, 
and hence $X_{1}u_{\mu}=X_{2}u_{\mu}$. 
Thus we have proved the lemma. 
\end{proof}

By \eqref{eq:nonzero} and \eqref{eq:equal}, 
we can define a bijection from $\CB_0(\mu)$ onto $\BB_0(\mu)$ by:
$Xu_{\mu} \mapsto X\pi_{\mu}$ for a monomial $X$ in Kashiwara operators; 
it is obvious that this map is an isomorphism of crystals. 
Thus we have proved Theorem \ref{thm:isom2}.

%
{\small
 \setlength{\baselineskip}{10pt}

}

\end{document}